\documentclass[12pt,a4paper]{article}
\usepackage{amsfonts}
\usepackage{mathrsfs}
\usepackage{amsmath}
\usepackage{color}
\usepackage{stmaryrd}
\usepackage{amssymb}
\usepackage{bbm}
\usepackage{graphicx}
\usepackage{enumerate}
\usepackage{theorem}
\usepackage{ulem}
\usepackage{authblk}
\makeatletter
\def\tank#1{\protected@xdef\@thanks{\@thanks
        \protect\footnotetext[0]{#1}}}
\def\bigfoot{

    \@footnotetext}
\makeatother

\newcommand{\ea}{\end{array}}
\newtheorem{theorem}{Theorem}[section]
\newtheorem{proposition}{Proposition}[section]

\newtheorem{remark}{Remark}[section]
\newtheorem{assumption}{Assumption}[section]

{\theorembodyfont{\rmfamily}
}

\newenvironment{proof}{Proof.}

\allowdisplaybreaks

\begin{document}
\title {\Large \bf Wong-Zakai approximations and support theorems for SDEs under Lyapunov conditions} 
\author[1]{Qi Li\thanks{E-mail:vivien777@mail.ustc.edu.cn}}
\author[1]{Jianliang Zhai\thanks{E-mail:zhaijl@ustc.edu.cn}}
\author[2]{Tusheng Zhang\thanks{E-mail:Tusheng.Zhang@manchester.ac.uk}}
\affil[1]{School of Mathematical Sciences, University of Science and Technology of China,
Hefei, Anhui 230026, China.}
\affil[2]{Department of Mathematics, University of Manchester, Oxford Road, Manchester, M13 9PL, UK.}
\renewcommand\Authands{ and }
\date{}
\maketitle
\begin{center}
\begin{minipage}{130mm}
{\bf Abstract.}
\hspace{1em}In this paper, we establish the Stroock-Varadhan type support theorems for stochastic differential equations (SDEs) under Lyapunov conditions, which significantly improve the existing results in the literature where the coefficients of the SDEs are required to be globally Lipschitz and of linear growth. Our conditions are very mild to include many important models, e.g. Threshold Ornstein-Ulenbeck process, Stochastic SIR model, Stochastic Lotka-Volterra systems, Stochastic Duffing-van der Pol oscillator model, which have polynomial the coefficients. To obtain the support theorem, we prove a new Wong-Zakai approximation problem, which is of independent interest.

\vspace{3mm} {\bf Keywords.}
Wong-Zakai approximation; support theorem; local Lipschitz; Lyapunov condition.
\end{minipage}
\end{center}

\section{Introduction}
\setcounter{equation}{0}
 \setcounter{definition}{0}
Let $W$ denote a $d$-dimentional standard Wiener process on a complete filtered probability space $(\Omega,\mathcal{F},(\mathcal{F}_t)_{t\geq0},\mathbb{P})$, where $(\mathcal{F}_t)_{t\geq0}$ is the normal filtration generated by $W$. Denote by $|\cdot|$, $\|\cdot\|$ and $\langle \cdot,\cdot\rangle$ the $\mathbb{R}^m$-norm, $\mathbb{R}^m \otimes \mathbb{R}^d$-norm and inner product of $\mathbb{R}^m$, respectively. Without loss of generality, we work on the finite time interval $[0, 1]$.
We denote by $\mathcal{C}([0,1];\mathbb{R}^m)$ the space of continuous functions $f:[0,1]\rightarrow \mathbb{R}^m$ with the norm $|f|_{\infty}:=\sup\limits_{t\in[0,1]}|f_t|$.
Let $\mathcal{H}$ denote the Cameron-Martin space, \textit{i.e.}, $\mathcal{H}:=\{ h:\dot{h}\in L^2([0,1];\mathbb{R}^d)\}$, where $\dot{h}$ denotes the derivative of $h$.


Consider the stochastic differential equation (SDE):
\begin{equation} 
\label{origin}
X_t = x + \int_0^t b(X_s)ds + \int_0^t \sigma(X_s)dW_s, \quad x\in \mathbb{R}^m , \quad t\in[0,1],
\end{equation}
where  $\sigma:\mathbb{R}^m \rightarrow \mathbb{R}^m \otimes \mathbb{R}^d$ and $b:\mathbb{R}^m \rightarrow \mathbb{R}^m$ are measurable functions. $b$, $\sigma$ and the derivative of $\sigma$, denoted by $\nabla\sigma$, are locally Lipschitz and satisfy some Lyapunov conditions; the precise assumptions on $b$ and $\sigma$ will be introduced in Section 3. In the sequel, we denote the solution of  (\ref{origin}) by $X:=(X_t, t\in[0,1])$.

The aim of this paper is to obtain the Stroock-Varadhan type support theorem for the SDE (\ref{origin}), that is, we characterize the support of $\mathbb{P} \circ X^{-1}$ as the closure of the set $\{S(h);h\in\mathcal{H}\}$ in $\mathcal{C}([0,1];\mathbb{R}^m)$, denoted by $\overline{\mathcal{S}}$, where $S(h)$ is the solution of the following deterministic equation:
\begin{equation}
\label{appro2}
S(h)_t=x+ \int_0^t [b(S(h)_s)-\frac{1}{2}(\nabla\sigma)\sigma(S(h)_s)]ds +\int_0^t \sigma(S(h)_s)\dot{h}_sds, \quad h\in \mathcal{H}.
\end{equation}
The support theorem is important in connection with the investigation of the accessibility, irreducibility and ergodicity of the Markov process generated by the solutions of SDE (\ref{origin}).

\vskip 0.3cm

The characterization of the topological support of the solutions was initially introduced by Stroock and Varadhan \cite{DS} for SDEs with bounded and globally Lipschitz coefficients. Gy\"{o}ngy and Pr\"{o}hle \cite{IT} later extended Stroock and Varadhan's result to SDEs with globally Lipschitz coefficients. Ben Arous, Gradinaru, and Ledoux \cite{GM, GMM} obtained the support theorems for SDEs in a $\alpha$-H\"{o}lder space, requiring  the coefficients $b$ and $\sigma$ to satisfy more strict conditions. Similar results as that in \cite{DS} are found in \cite{SSD}.
We particularly want to mention the reference \cite{AM} in which A. Millet and M. Sanz-Solé proposed a simple, effective approach to obtain the support theorems for SDEs by proving Wong-Zakai type approximations for some associated SDEs.

In the work on support theorems so far, for technical reasons people always assume that the coefficients of the SDEs are globally Lipschitz and of linear growth. These restrictions exclude many important models, like Threshold Ornstein-Ulenbeck process, Stochastic SIR models, Stochastic Lotka-Volterra systems, Stochastic Duffing-van der Pol oscillator models, where the coefficients are of polynomial growth. The purpose of this paper is to
 extend the Stroock-Varadhan support theorem to SDEs under Lyapunov conditions. Our conditions are very mild to allow coefficients of the SDEs to be locally Lipschitz with polynomial growth. In particular, the results can be applied to the interesting models mentioned above.

To obtain our main results, we adopt the same approach as that in \cite{AM}. The crucial step is to prove a Wong-Zakai type approximation for SDEs with local Lipschitz coefficients satisfying certain Lyapunov conditions. The idea is to introduce some localization arguments in order to utilize the existing results in the case where the coefficients of the SDEs are bounded and globally Lipschitz. The Wong-Zakai approximation itself is of independent interest.


Before ending the introduction, let us briefly mention some relevant work on Wong-Zakai approximation of SDEs. Wong-Zakai approximation was introduced by Wong and Zakai in their pioneer work \cite{WZ,WZ2} based on the idea that the noise in SDEs can be approximated by piecewise linear approximations of Brownian motion. Since then, there are a number of papers devoted to this topic. We mention the relevant  work \cite{CK,GX,KA,K,M,AM,NY,RW,SI,DS} and references therein. However, in all these works, the coefficients of the SDEs to required to be globally Lipschitz and bounded.


 We would like to point out that apart from the application of Wong-Zakai approximations to prove the Stroock-Varadhan support theorems, it can be employed to derive some numerical approximation schemes for SDEs, which find many applications in electrical engineering, energy engineering, and related fields, see \cite{SZM,WEG,MN}.


The paper is organized as follows. In Section 2, we introduce the precise assumptions of $B,H,F$ and $G$ and prove the Wong-Zakai approximation results. In Section 3, we prove the support theorem for SDEs. Some new applications are presented in Section 4.

\section{Wong-Zakai approxiamtions}
\setcounter{equation}{0}
 \setcounter{definition}{0}
 In this section we will establish the Wong-Zakai approximations for SDEs with locally Lipschitz coefficients.
 Given a positive integer $n$, for $t \in[0,1],\frac{k}{2^n}\leq t < \frac{k+1}{2^n}$, set
\begin{equation}
\text{\uwave{t}}_n = \frac{k}{2^n}, \quad \underline{t}_n=\frac{k-1}{2^n}\vee 0,
\end {equation}
and define
\begin{equation}
\label{wn}
W_t^n = W_{\uline{t}_n}+ 2^n(t- \text{\uwave{t}}_n )\left[W_{\text{\uwave{t}}_n } -W_{\underline{t}_n}\right].
\end{equation}

Let $B:\mathbb{R}^m \rightarrow \mathbb{R}^m $, $H,G$ and $F: \mathbb{R}^m \rightarrow \mathbb{R}^m \otimes \mathbb{R}^d$ be given measurable mappings.\\
Introduce the following conditions:

\begin{assumption} \label{assump1}
Assume $B, H, F$ are continuous maps on $\mathbb{R}^m$ and $G$ is a $\mathcal{C}^1$-map on $\mathbb{R}^m$ satisfying:

$(i)$ $B$, $H, F$, $G$ and $\nabla G$ are locally Lipschitz.

$(ii)$ There exist a Lyapunov function $V \in \mathcal{C}^2(\mathbb{R}^m;\mathbb{R}_+)$ and $\theta>0,\eta>0$ such that
\begin{equation}
\begin{aligned}\nonumber
&\hspace{-1cm}(a). \lim\limits_{|x|\rightarrow + \infty}V(x)=+\infty,\\
&\hspace{-1cm}(b). J_1(x):=\ \langle B(x),\nabla V(x)\rangle \\
&\hspace{-1cm}\hspace{4.5em} + \frac{\theta}{2}Trace\{(H(x)^\ast\nabla^2V(x)H(x))+(G(x)^\ast\nabla^2V(x)G(x)) +(F(x)^\ast\nabla^2V(x)F(x))\} \\
&\hspace{-1cm}\hspace{4.5em} + \frac{|(H(x)+G(x)+F(x))^\ast\nabla V(x)|^2}{\eta V(x)}\\
&\hspace{-1cm}\hspace{4em}\leq\  C(1+V(x)),\\
&\hspace{-1cm}(c). Trace\{(H(x)^\ast\nabla^2V(x)H(x))+(G(x)^\ast\nabla^2V(x)G(x))+(F(x)^\ast\nabla^2V(x)F(x))\}\geq -M-CV(x).\\
\quad \text{and}\\
&\hspace{-1cm}(d). J_2(x):=\ \langle B(x)+\nabla G(x)[F(x)+\frac{1}{2}G(x)],\nabla V(x)\rangle \\
&\hspace{-1cm}\hspace{4.5em}+ \frac{\theta}{2}Trace\{(H(x)^\ast\nabla^2V(x)H(x))+((F(x)+G(x))^\ast\nabla^2V(x)(F(x)+G(x)))\} \\
&\hspace{-1cm}\hspace{4.5em} + \frac{|(H(x)+G(x)+F(x))^\ast\nabla V(x)|^2}{\eta V(x)}\\
&\hspace{-1cm}\hspace{4em}\leq \ C(1+V(x)),\\
&\hspace{-1cm}(e). Trace\{(H(x)^\ast\nabla^2V(x)H(x)) +((F(x)+G(x))^\ast\nabla^2V(x)(F(x)+G(x)))\}\geq -M-CV(x).\\
\end{aligned}
\end{equation}
Here $\nabla V$, $\nabla^2 V$ and $\nabla G$ stand for the gradient vector, Hessian matrix of the function $V$ and the derivative of the function $G$, respectively; $\cdot^\ast$ denotes the transpose of the matrix; $C$, $M>0$ are some fixed constants.
\end{assumption}
For  $h\in \mathcal{H}$, consider SDEs:
 \begin{align} 
\label{Yn}
Y_t^n &= x + \int_0^tB(Y_s^n)ds + \int_0^t H(Y_S^n)\dot{h}_sds + \int_0^t G(Y_s^n)\dot{W}_s^nds + \int_0^t F(Y_s^n)dW_s,\\
\nonumber
Z_t &= \  x + \int_0^t B(Z_s)ds + \int_0^t H(Z_s)\dot{h}_sds \\
\label{Z}
&\quad\  + \int_0^t \nabla G(Z_s)[F(Z_s)+\frac{1}{2}G(Z_s)]ds + \int_0^t [F(Z_s)+G(Z_s)]dW_s.\\
\nonumber
\end{align}
It is known (see e.g. \cite{JH} Propositions 2.1 and  2.2) that under Assumption \ref{assump1} the SDEs (\ref{Yn}) and (\ref{Z}) admit unique solutions.
The following result is the Wong-Zakai approximation.
\begin{theorem}
Let $Y^n$, $Z$ be solutions of SDEs (\ref{Yn}) and (\ref{Z}), respectively. Suppose Assumption \ref{assump1} is in place. Then, for $\delta>0$
 \begin{equation}
\label{3}
\lim\limits_{n\rightarrow \infty}\mathbb{P}(|Y^n-Z|_{\infty}>\delta)=0.\\
\end{equation}
\end{theorem}

\begin{proof}
Without loss of generality, we assume $\delta<1$. For $R\in\mathbb{N}$, set
\begin{equation}  
\mathbb{B}_R=\{f: f\in \mathcal{C}([0,1];\mathbb{R}^m),|f|_{\infty}\leq R\}.\\
\nonumber
\end{equation}
Since $Z\in \mathcal{C}([0,1];\mathbb{R}^m)$ a.e., we have,

\begin{equation} 
\label{4}
1=\mathbb{P}(Z(\omega)\in \bigcup\limits_{R\in\mathbb{N}}\mathbb{B}_R)= \lim\limits_{R\rightarrow \infty}\mathbb{P}(Z(\omega)\in \mathbb{B}_R).\\
\end{equation}
Let $\Omega_R:=\{\omega \in \Omega: Z(\omega)\in \mathbb{B}_R\}$. In view of (\ref{4}),  for any given $\epsilon>0$, we can choose
$R\in \mathbb{N}$ sufficiently large  so that
\begin{equation} 
\label{5}
\mathbb{P}(\Omega_R^c)=1-\mathbb{P}(\Omega_R)<\frac{\epsilon}{2}.\\
\end{equation}
Take a smooth truncation function $\theta_R:\mathbb{R}^m\rightarrow \mathbb{R}$ such that $\theta_R\in [0,1]$ and that
\begin{equation}
\theta_R =\left\{
\begin{aligned}
&1 \quad \text{if }|x|\leq R+1,\\
&0 \quad \text{if }|x|\geq 2(R+1).\\
\end{aligned}
\right.
\end{equation}
Set $B_R(x)=\theta_R(x) B(x), H_R(x)=\theta_R(x) H(x), G_R(x)=\theta_R(x) G(x), F_R(x)=$
$\theta_R(x) F(x)$.

It's easy to see that

$(i)$ $B_R, H_R, G_R, F_R$ and $\nabla G_R$ are globally Lipschitz and bounded.

$(ii)$ $B_R, H_R, G_R, F_R$ and $\nabla G_R$ coincide with $B, H, G, F$ and $\nabla G$ on the closed ball $\overline{B}(0,R+1)\subset \mathbb{R}^m$.

For $R>0$, $h\in \mathcal{H}$, introduce the following SDEs:

\begin{align} 
\label{6}
Z_t^R = & \  x + \int_0^t B_R(Z_s^R)ds + \int_0^t H_R(Z_s^R)\dot{h}_sds \\
\nonumber
&+ \int_0^t \nabla G_R(Z_s^R)[F(Z_s)+\frac{1}{2}G_R(Z_s^R)]ds + \int_0^t [F_R(Z_s^R)+G_R(Z_s^R)]dW_s.\\
Y_t^{n,R} = & \ x + \int_0^tB_R(Y_s^{n,R})ds + \int_0^t H_R(Y_S^{n,R})\dot{h}_sds + \int_0^t G_R(Y_s^{n,R})\dot{W}_s^nds +
\label{7}
 \int_0^t F_R(Y_s^{n,R})dW_s.\\
\nonumber
\end{align}
Then, since $B_R, H_R, G_R, F_R$ and $\nabla G_R$ are globally Lipschitz and bounded, according to Theorem 3.5 in \cite{AM}, we have
\begin{equation}\label{9-1}
\lim_{n\rightarrow\infty}\mathbb{P}(|Z^R-Y^{n,R}|_{\infty}>\delta)=0.
\end{equation}
For $\omega \in \Omega_R$, we have
\begin{equation}
|Z(\omega,\cdot)|_{\infty}\leq R.\\
\nonumber
\end{equation}
This indicates that on $\Omega_R$, $Z$ satisfies the same equation (\ref{6}) as $Z^R$. The uniqueness of solutions to equation (\ref{6}) implies that $Z^R(\omega,t)=Z(\omega,t)$ on $[0,1]$ for $\omega \in \Omega_R$.

For $R>0$, in view of (\ref{9-1}), there exists $N_{\epsilon,R}\in \mathbb{N}$ such that for all $n\geq N_{\epsilon,R}$,
\begin{equation} 
\label{P}
\mathbb{P}(|Z^R-Y^{n,R}|_{\infty}>\delta)<\frac{\epsilon}{2}.\\
\end{equation}
For $n\geq N_{\epsilon,R}$, set
\begin{equation}
\label{Ond}
\Omega_{n,R}:=\{\omega\in\Omega:|Z^R(\omega,\cdot)-Y^{n,R}(\omega,\cdot)|_{\infty}>\delta\}.\\
\vspace{2em}
\end{equation}
We claim that if $n\geq N_{\epsilon,R}$ and if $\omega \in (\Omega_R^c\cup \Omega_{n,R})^c=\Omega_R\cap\Omega_{n,R}^c$, then
\begin{equation}
Y^{n,R}(\omega,t)=Y^n(\omega,t) \quad \text{on } [0,1].
\label{10}
\end{equation}
Indeed, for $\omega \in \Omega_R$, we have $|Z^R(\omega,\cdot)|_{\infty}=|Z(\omega,\cdot)|_{\infty}\leq R$.
If $\omega \in \Omega_{n,R}^c$, we have $|Z^R(\omega,\cdot)-Y^{n,R}(\omega,\cdot)|_{\infty}\leq\delta<1$.
Therefore, on the set  $\Omega_R\cap\Omega_{n,R}^c$,
\begin{equation}
\label{11}  
|Y^{n,R}(\omega,\cdot)|_{\infty}\leq |Z^R(\omega,\cdot)-Y^{n,R}(\omega,\cdot)|_{\infty} + |Z^R(\omega,\cdot)|_{\infty}\leq R+1.
\end{equation}
For $R>0$, define the stopping times:
\begin{align}
&\tau^{R+1}=\text{inf}\{t, |Y^n_t|\geq R+1\},\\
&\tilde{\tau}^{R+1}=\text{inf}\{t, |Y^{n,R}_t|\geq R+1\}.
\end{align}
Set $\tau=\tau^{R+1}\wedge \tilde{\tau}^{R+1}$. Then,  $Y^n_{t\wedge \tau}$ and $Y^{n,R}_{t\wedge \tau}$ are solutions of
the same equation (\ref{7}). Hence by the uniqueness,
\begin{equation}    
\label{8}
Y^n(\omega,t\wedge \tau) = Y^{n,R}(\omega,t\wedge \tau),\quad \text{on } [0,1].
\end{equation}
This further implies that  $\tau=\tau^{R+1}= \tilde{\tau}^{R+1}$.
By (\ref{11}), we see that
\begin{equation}\label{8-1}
\Omega_R\cap\Omega_{n,\delta}^c \subset \{\tilde{\tau}^{R+1}>1\}=\{\tau>1\}=\{\tau^{R+1}>1\}.
\end{equation}
Combing (\ref{8}) with (\ref{8-1}) yields that for $\omega \in (\Omega_R^c\cup \Omega_{n,\delta})^c=\Omega_R\cap\Omega_{n,\delta}^c$,
\begin{equation}\label{8-2}
Y^{n,R}(\omega,t)=Y^n(\omega,t) \quad \text{on } [0,1].
\end{equation}
Combining (\ref{10}) and (\ref{Ond}) together, we deduce that for $n\geq N_{\epsilon,R}$, $\omega \in (\Omega_R^c\cup \Omega_{n,\delta})^c=\Omega_R\cap\Omega_{n,\delta}^c$,.
\begin{equation}
|Y^n(\omega,\cdot)-Z(\omega,\cdot)|_{\infty}=|Y^{n,R}(\omega,\cdot)-Z^R(\omega,\cdot)|_{\infty}\leq \delta.
\end{equation}
Consequently, by (\ref{5}), (\ref{P}),
\begin{equation}
\begin{aligned}
\mathbb{P}(|Y^n-Z|_{\infty}>\delta)&\leq\ \mathbb{P}(\Omega_R^c\cup \Omega_{n,\delta})\\
                                   &\leq\ \mathbb{P}(\Omega_R^c) + \mathbb{P}(\Omega_{n,\delta})\\
                                   &\leq\  \frac{\epsilon}{2} + \frac{\epsilon}{2} =\ \epsilon.
\nonumber
\end{aligned}
\end{equation}
Since $\epsilon$ is arbitrary, we conclude that
$$\lim_{n\rightarrow\infty}\mathbb{P}(|Y^n-Z|_{\infty}>\delta)=0$$
completing the proof.
\end{proof}

\section{Support theorems}
\setcounter{equation}{0}
 \setcounter{definition}{0}
 Now we turn back to the stochastic differential equation (\ref{origin}) and present a support theorem for the solution. Regarding the coefficients $b$, $\sigma$, we introduce the following condition.
\begin{assumption} \label{assump2}
Assume $b$ is a continuous map on $\mathbb{R}^m$ and $\sigma$ is a $\mathcal{C}^1$-map on $\mathbb{R}^m$ satisfying:

$(i)$ Both b, $\sigma$ and $\nabla \sigma$ are locally Lipschitz.

$(ii)$ There exist a Lyapunov function $V \in \mathcal{C}^2(\mathbb{R}^m;\mathbb{R}_+)$ and $\theta>0,\eta>0$ such that
\begin{equation}\nonumber
\lim\limits_{|x|\rightarrow + \infty}V(x)=+\infty,
\end{equation}
\begin{equation}\nonumber
J_1(x):=\langle b(x),\nabla V(x)\rangle + \frac{\theta}{2}Trace(\sigma^\ast(x)\nabla^2V(x)\sigma(x)) + \frac{|\sigma^\ast(x)\nabla V(x)|^2}{\eta V(x)}\leq C(1+V(x)),\\
\end{equation}
\begin{equation}\nonumber
J_2(x):=\langle b(x)-\frac{1}{2}\nabla\sigma\sigma(x),\nabla V(x)\rangle + \frac{\theta}{2}Trace(\sigma^\ast(x)\nabla^2V(x)\sigma(x)) + \frac{|\sigma^\ast(x)\nabla V(x)|^2}{\eta V(x)}\leq C(1+V(x)),
\end{equation}
and\\
\begin{equation}\nonumber
Trace(\sigma^\ast(x)\nabla^2V(x)\sigma(x))\geq -M-CV(x).
\end{equation}
Here $\nabla V$, $\nabla^2 V$ and $\nabla \sigma$ stand for the gradient vector, Hessian matrix of the function $V$ and the derivative of the function $\sigma$, respectively; $\sigma^\ast(x)$ denotes the transpose of $\sigma(x)$; $C$, $M>0$ are some fixed constants.
\end{assumption}
\begin{remark}

\item[1).] The assumption on $J_2(x)$ is proposed to ensure the existence of a unique solution to equation (\ref{appro2}).
\item[2).] The Lyapunov conditions mentioned above are not actually used in our proof. Indeed, once the equation(\ref{origin}) possesses a unique solution and the truncated coefficients are globally Lipschitz and bounded, the Wong-Zakai and support results can be established in a similar way.

\end{remark}
Recall the stochastic differential equation:
\begin{equation} 
\label{3-1}
X_t = x + \int_0^t b(X_s)ds + \int_0^t \sigma(X_s)dWs, \quad t\in[0,1],
\end{equation}
The existence and uniqueness of the solution $X$ under the Assumption \ref{assump2} follows from \cite{JH}.\\

%

The aim of this section is to characterize the support of $\mathbb{P} \circ X^{-1}$ as the closure $\mathcal{S}$ of the set $\{S(h);h\in\mathcal{H}\}$ in $\mathcal{C}([0,1];\mathbb{R}^m)$.

First we recall the following Proposition from \cite{AM}.
\begin{proposition}
Consider a measurable map $V:\Omega\rightarrow E$, where $(E,|\cdot|_E)$ is a seperable Banach space.

$(\mathsf{1})$ let $\zeta_1:\mathcal{H}\rightarrow E$ be a measurable map, and let $H_n:\Omega\rightarrow \mathcal{H}$ be a sequence of random

variable such that for any $\epsilon>0$,
\begin{equation}
\lim\limits_{n\rightarrow \infty}\mathbb{P}(|V(\omega)-\zeta_1(H_n(\omega))|_E>\epsilon)=0.\\
\end{equation}

Then
\begin{equation}
\text{support}(\mathbb{P}\circ V^{-1})\subset \overline{\zeta_1(\mathcal{H})}.
\end{equation}

$(\mathsf{2})$ Let $\zeta_2:\mathcal{H}\rightarrow E$ be a map, and for fixed $h$ let $T_n^h:\Omega \rightarrow \Omega$ be a sequence of measurable

transformations such that $\mathbb{P}\circ(T_n^h)^{-1}\ll \mathbb{P}$, and for any $\epsilon >0$,
\begin{equation}
\limsup\limits_{n\rightarrow \infty} \mathbb{P}(|V(T_n^h(\omega))-\zeta_2(h)|_E<\epsilon)>0.
\end{equation}

Then support$(\mathbb{P}\circ V^{-1})\supset \overline{\zeta_2(\mathcal{H})}$.
\end{proposition}

The following result is the support theorem.

\begin{theorem} 
\label{support}
Suppose Assumption $\ref{assump2}$ holds. Let $X$ and $S(h)$ be the solutions to equations $\left( \ref{3-1}\right)$ and $\left(\ref{appro2}\right)$. Then supp$(\mathbb{P}\circ X^{-1})=\overline{\mathcal{S}}$, where $\overline{\mathcal{S}}$ denotes the closure of $\mathcal{S}=\{S(h);h\in\mathcal{H}\}$ in the space $\mathcal{C}([0,1];\mathbb{R}^m)$ and supp$(\mathbb{P}\circ X^{-1})$ denotes the support of the distribution $\mathbb{P}\circ X^{-1}$.
\end{theorem}
\begin{proof}
We will apply Proposition 3.1. To this end, we take $E=\mathcal{C}([0,1];\mathbb{R}^m)$, $V=X$, $\zeta_1=\zeta_2=S(\cdot)$, $H_n(\omega)= \omega^n$ and $T_n^h(\omega)=\omega-\omega^n+h$, where $\omega^n$ is defined as (\ref{wn}). Then Girsanov's theorem implies that $\mathbb{P}\circ(T_n^h)^{-1}\ll \mathbb{P}$. Thus, according to Proposition 3.1, the equality supp$(\mathbb{P}\circ X^{-1})=\overline{\mathcal{S}}$ will follow from the following approximation results, for every $\delta>0$:
\begin{align} 
\label{1}
&\lim\limits_{n\rightarrow \infty}\mathbb{P}(|X(\omega)-S(\omega^n)|_{\infty}>\delta)=0,\\
&\lim\limits_{n\rightarrow \infty} \mathbb{P}(|X(\omega-\omega^n+h)-S(h)|_{\infty}>\delta)>0.
\label{2}
\end{align}
Where $S(\omega^n)$ is the solution of equation
\begin{equation}\label{001}
S(\omega^n)_t= x +\int_0^t [b(S(\omega^n)_s)-\frac{1}{2}(\nabla\sigma)\sigma(S(\omega^n)_s)]ds +\int_0^t \sigma(S(\omega^n)_s)\dot{W}^n_sds.\\
\end{equation}
On the other hand, approximations of stochastic integrals by Riemann sums imply that $X^n(\omega):=X(\omega-\omega^n+h)$ is the solution of the following stochastic differential equation:
\begin{equation}\label{002}
X^n_t=x +\int_0^t b(X_s^n)ds + \int_0^t \sigma(X_s^n)\dot{h}_sds - \int_0^t \sigma(X_s^n)\dot{W}_s^nds + \int_0^t \sigma(X_s^n)dW_s.
\end{equation}
By a close examination, we find that both $X^n$ and $ S(\omega^n)$ are particular cases of solutions of the stochastic differential equation (\ref{Yn}) in Section 2.
Actually, setting $B=b-\frac{1}{2}(\nabla\sigma)\sigma,H=0, G=\sigma$ and $F=0$ in the equation (\ref{3}) we obtain (\ref{1}); while setting $B=b,H=\sigma, G=-\sigma$ and $F=\sigma$ gives (\ref{2}).
Therefore,  (\ref{1}) and (\ref{2}) are the particular cases of the convergence stated in Theorem 2.1:
\begin{equation}
\lim\limits_{n\rightarrow \infty}\mathbb{P}(|Y^n-Z|_{\infty}>\delta)=0.
\end{equation}
The proof is complete.
\end{proof}

\section{Examples}
\setcounter{equation}{0}
 \setcounter{definition}{0}
The assumption $\ref{assump2}$ is very mild to include many interesting models. In this section we provide some examples to which the main results apply.

\textbf{Example 4.1} Consider the following one-dimensional SDE:
\begin{equation}
dx_t = -x^3_tdt + x^2_t dB_t.
\end{equation}

In this case, $b(x)=-x^3$, $\sigma(x)=x^2$ and $\nabla\sigma(x)=2x$.
If we take $\theta=1,\eta=4$ and $V(x)=x^2$, then
\begin{equation}\nonumber
\begin{aligned}\nonumber
&\lim\limits_{|x|\rightarrow + \infty}V(x)=+\infty,\\
\nonumber
&J(x):=\langle b(x),\nabla V(x)\rangle + \frac{\theta}{2}Trace(\sigma^\ast(x)\nabla^2V(x)\sigma(x)) + \frac{|\sigma^\ast(x)\nabla V(x)|^2}{\eta V(x)}=0 \leq 1+V(x),\\
\nonumber
&\langle-\frac{1}{2}\nabla\sigma\sigma,\nabla V(x)\rangle \leq0,\\
\nonumber
&Trace(\sigma^\ast(x)\nabla^2V(x)\sigma(x))=2x^4 \geq -V(x).\\
\nonumber
\end{aligned}
\end{equation}
This shows that Assumption \ref{assump2} holds. We can now apply Theorem \ref{support} to get the following result.
\begin{proposition}
Let $\mathcal{H}$ denote the Cameron-Martin space. For $h\in \mathcal{H}$, let $S(h)$ be the solution of the differential equation:
\begin{equation}
S(h)_t=x_0+ \int_0^t -2S(h)_s^3 ds +\int_0^t S(h)_s^2 \dot{h}_sds.\\
\end{equation}\nonumber
 Then supp$(\mathbb{P}\circ x^{-1})=\overline{\mathcal{S}}$, where $\overline{\mathcal{S}}$ denotes the closure of $\mathcal{S}=\{S(h);h\in\mathcal{H}\}$ in $\mathcal{C}([0,1];\mathbb{R})$.
\end{proposition}
\vskip 0.3cm
The next three examples  are taken from  \cite{JH}.\\
\vskip 0.3cm
\textbf{Example 4.2}(Stochastic Duffing-van der Pol oscillator model) The Duffing-van der Pol oscillator equation unifies both the Duffing equation and the van der Pol equation describing a self-oscillating triode/diode circuit. The stochastic version of the model is given by the following SDE (see \cite{SMA}).
\begin{equation}
\begin{aligned}
\ddot{X}^{x,1}_t &= \alpha_2\dot{X}^{x,1}_t - \alpha_1 X^{x,1}_t-\alpha_3(X^{x,1}_t)^2\dot{X}^{x,1}_t - (X^{x,1}_t)^3 + g(X^{x,1}_t)\dot{W}_t,\\
X^{x,1}_0        &= x_1, \quad \dot{X}^{x,1}_0=x_2,
\nonumber
\end{aligned}
\end{equation}
where $\alpha_1,\alpha_2,\alpha_3\in (0,\infty).$ Here we assume $|g(x)|^2\leq\eta_0+\eta_1|x|^4,\eta_0,\eta_1>0$ and both $g(x)$ and $g'(x)$ are locally Lipschitz.
Setting $X^{x,2}_t:=\dot{X}^{x,1}_t$, then the above equation is equivalent to following system of SDEs:
\begin{equation}
\begin{aligned}
dX^{x,1}_t &= X^{x,2}_tdt,\\
dX^{x,2}_t &= [\alpha_2 X^{x,2}_t - \alpha_1 X^{x,1}_t-\alpha_3(X^{x,1}_t)^2 X^{x,2}_t - (X^{x,1}_t)^3] +g(X^{x,1}_t)dW_t,\\
X^{x,1}_0  &=x_1, \quad X^{x,2}=x_2.
\nonumber
\end{aligned}
\end{equation}
For $x=(x_1,x_2)\in\mathbb{R}^2$, set $b(x)=(x_2,\alpha_2x_2-\alpha_1x_1-\alpha_3(x_1)^2x_2-(x_1)^3)^T$ and $\sigma(x)=(0,g(x_1))^T$.
Define $V(x)=\frac{(x_1)^4}{2}+\alpha_1(x_1)^2 +(x_2)^2, \theta=\eta=1$. Then
\begin{equation}\nonumber
\lim\limits_{|x|\rightarrow + \infty}V(x)=+\infty.\\
\end{equation}
\begin{equation}
\begin{aligned}\nonumber
J(x):&=\langle b(x),\nabla V(x)\rangle + \frac{1}{2}Trace(\sigma^\ast(x)\nabla^2V(x)\sigma(x)) + \frac{|\sigma^\ast(x)\nabla V(x)|^2}{V(x)}\\
&=x_2(2(x_1)^3 +2\alpha_1x_1) + 2x_2(\alpha_2 x_2-\alpha_1x_1-\alpha_3(x_1)^2x_2-(x_1)^3)\\
& \quad +|g(x_1)|^2 +\frac{4(x_2)^2|g(x_1)|^2}{\frac{(x_1)^4}{2}+\alpha_1(x_1)^2 +(x_2)^2}\\
&=2\alpha_2(x_2)^2-2\alpha_3(x_1)^2(x_2)^2 + |g(x_1)|^2 + \frac{4(x_2)^2|g(x_1)|^2}{\frac{(x_1)^4}{2}+\alpha_1(x_1)^2 +(x_2)^2}\\
&\leq\eta_0 + 2\alpha_2(x_2)^2 + \eta_1(x_1)^4 + \frac{4\eta_0(x_2)^2 + \eta_1(x_1)^4(x_2)^2}{\frac{(x_1)^4}{2}+\alpha_1(x_1)^2 +(x_2)^2}\\
&\leq 5\eta_0 + (8\eta_1 +2\alpha_2)(x_2)^2+\eta_1(x_1)^4\\
&\leq (5\eta_0+10\eta_1+2\alpha_2)(1+V(x)),\\
\langle-\frac{1}{2}\nabla&\sigma\sigma,\nabla V(x)\rangle =0,
\end{aligned}
\end{equation}
and
\begin{equation}\nonumber
Trace(\sigma^\ast(x)\nabla^2V(x)\sigma(x))=2|g(x_1)|^2 \geq -V(x).\\
\end{equation}
Hence Assumption \ref{assump2} holds and we have the following result.
\begin{proposition}
Let $\mathcal{H}$ denote the Cameron-Martin space. For $h\in \mathcal{H}$, let $S(h)=(S(h)_1,S(h)_2)^T$ be the solution of the differential equation:
\begin{equation}\nonumber
\begin{aligned}
S(h)_{1,t}&=x_1+ \int_0^t S(h)_{2,s} ds.\\
S(h)_{1,t}&=x_2+ \int_0^t \alpha_2 S(h)_{2,s}- \alpha_1 S(h)_{1,s}-\alpha_3(s(h)_{1,s})^2S(h)_{2,s}- (S(h)_{1,s})^3 ds +\int_0^t g(S(h)_{1,s})\dot{h}_sds.\\
\end{aligned}
\end{equation}\nonumber
 Then supp$(\mathbb{P}\circ X^{-1})=\overline{\mathcal{S}}$, where $\overline{\mathcal{S}}$ denotes the closure of $\mathcal{S}=\{S(h);h\in\mathcal{H}\}$ in $\mathcal{C}([0,1];\mathbb{R}^2)$.
\end{proposition}
\vskip 0.3cm 
\textbf{Example 4.3}(Stochastic Lotka-Volterra(LV) systems) The Lotka-Volterra systems play an important role in game theory, population dynamics etc.(see \cite{JK}). Here we consider the three-dimensional Stratonovich stochastic copetitive LV system:
\begin{equation}\nonumber
\begin{aligned}
dy_1 &= y_1(r-a_{11}y_1-a_{12}y_2-a_{13}y_3)dt + \gamma y_1\circ dB_t.\\
dy_2 &= y_2(r-a_{21}y_1-a_{22}y_2-a_{23}y_3)dt + \gamma y_2\circ dB_t.\\
dy_3 &= y_3(r-a_{31}y_1-a_{32}y_2-a_{33}y_3)dt + \gamma y_3\circ dB_t.\\
\end{aligned}
\end{equation}
where $\sigma,r>0,a_{ij}>0,i,j=1,2,3;\gamma$ is parameter and initial data$(y_1(0),y_2(0),y_3(0))\in(0,+\infty)^3$. According to \cite{LZ} Thm 3.2, we know $y(t)=(y_1(t),y_2(t),y_3(t))\in(0,+\infty)^3$ for all $t>0$. And the above system is equivalent to the It$\hat{o}$ stochastic Lotka-Volterra system:
\begin{equation}\nonumber
dy_i = y_i(r + \frac{\gamma^2}{2} -\sum_{j=1}^3a_{ij}y_j)dt + \gamma y_idB_t,\quad i=1,2,3.\\
\end{equation}
Set $b(y)=(y_1(r + \frac{\gamma^2}{2} -\sum_{j=1}^3a_{1j}y_j),y_2(r + \frac{\gamma^2}{2} -\sum_{j=1}^3a_{2j}y_j),y_3(r + \frac{\gamma^2}{2} -\sum_{j=1}^3a_{3j}y_j))^T,\sigma(y)=(\gamma y_1,\gamma y_2,\gamma y_3)^T$.
Let $V(y)=|y|^2$.  Then
\begin{equation}\nonumber
\begin{aligned}
&\lim\limits_{|y|\rightarrow + \infty}V(y)=+\infty.\\
&J(y):=\langle b(y),\nabla V(y)\rangle + \frac{\theta}{2}Trace(\sigma^\ast(y)\nabla^2V(y)\sigma(y)) + \frac{|\sigma^\ast(y)\nabla V(y)|^2}{\eta V(y)}\\
&\hspace{2.2em} =\sum_{i=1}^3 2y_i^2(r + \frac{\gamma^2}{2} -\sum_{j=1}^3a_{ij}y_j)+\theta\gamma^2V(y) +\frac{4\gamma^2}{\eta}V(y)\\
&\hspace{2.2em} \leq 2V(y)(r + \frac{\gamma^2}{2}) + \theta\gamma^2V(y) +\frac{4\gamma^2}{\eta}V(y)\\
&\hspace{2.2em}\leq C(1+V(y)),\\
&\langle-\frac{1}{2}\nabla\sigma\sigma,\nabla V(x)\rangle \leq0,\\
&Trace(\sigma^\ast(y)\nabla^2V(y)\sigma(y))=2\gamma^2V(y)\geq 0.\\
\end{aligned}
\end{equation}
Hence, Assumption \ref{assump2} holds. Apply Theorem \ref{support} to get
\begin{proposition}
Let $\mathcal{H}$ denote the Cameron-Martin space. For $h\in \mathcal{H}$, let $S(h)=(s(h)_1,S(h)_2),$\\
$S(h)_3)^T$ be the solution of the differential equation:
\begin{equation}\nonumber
\begin{aligned}
S(h)_{1,t}&=y_1(0)+ \int_0^t S(h)_{1,s}(r  -\sum_{j=1}^3a_{1j}S(h)_{j,s})ds + \int_0^t\gamma S(h)_{1,s}\dot{h}_sds.\\
S(h)_{2,t}&=y_2(0)+ \int_0^t S(h)_{2,s}(r  -\sum_{j=1}^3a_{2j}S(h)_{j,s})ds + \int_0^t\gamma S(h)_{2,s}\dot{h}_sds.\\
S(h)_{3,t}&=y_2(0)+ \int_0^t S(h)_{3,s}(r  -\sum_{j=1}^3a_{3j}S(h)_{j,s})ds + \int_0^t\gamma S(h)_{3,s}\dot{h}_sds.\\
\end{aligned}
\end{equation}
 Then supp$(\mathbb{P}\circ X^{-1})=\overline{\mathcal{S}}$, where $\overline{\mathcal{S}}$ denotes the closure of $\mathcal{S}=\{S(h);h\in\mathcal{H}\}$ in $\mathcal{C}([0,1];\mathbb{R}^3)$.
\end{proposition}
\textbf{Example 4.4}(Stochastic SIR model) The SIR model from epidemiology for the total number of susceptible, infected and and revovered individuals has been introduced by Anderson and May \cite{RM2}.
Here we consider the following stochatic SIR model:
\begin{equation}\nonumber
\begin{aligned}
dX_t^{x,1} &= (-\alpha X_t^{x,1}X_t^{x,2}-\kappa X_t^{x,1} +\kappa)dt -\beta X_t^{x,1}X_t^{x,2}dW_t,\\
dX_t^{x,2} &= (\alpha X_t^{x,1}X_t^{x,2}-(\gamma +\kappa) X_t^{x,2})dt +\beta X_t^{x,1}X_t^{x,2}dW_t,\\
dX_t^{x,3} &= (\gamma X_t^{x,2}-\kappa X_t^{x,3})dt,\\
X_0^{x,1} &= x_1,\quad X_0^{x,2} = x_2,\quad X_0^{x,3} = x_3,\\
\end{aligned}
\end{equation}
where $\alpha,\beta,\gamma,\kappa\in (0,\infty)$ and $x=(x_1,x_2,x_3)\in[0,\infty)^3$.

For $x=(x_1,x_2,x_3)\in[0,\infty)^3$, set $b(x)=(-\alpha x_1x_2-\kappa x_1 +\kappa, \alpha x_1x_2-(\gamma + \kappa)x_2,\gamma x_2-\kappa x_3)^T$ and $\sigma(x)=(-\beta x_1x_2,\beta x_1x_2,0)^T.$ It's easy to see that $b$ and $\sigma$ are local Lipschitz continuous and satisfy the Assumption \ref{assump2} with $V(x)=(x_1 + x_2 - 1)^2$ for any positive $\theta$ and $\eta$. Then applying Thm \ref{support}, we have:
\begin{proposition}
Let $\mathcal{H}$ denote the Cameron-Martin space. For $h\in \mathcal{H}$, let $S(h)=(s(h)_1,S(h)_2,$\\
$S(h)_3)^T$ be the solution of the differential equation:
\begin{equation}\nonumber
\begin{aligned}
S(h)_{1,t}&=x_1+ \int_0^t (-\alpha S(h)_{1,s}S(h)_{2,s}-\frac{\beta^2}{2}(S(h)_{1,s}S(h)_{2,s})(S(h)_{2,s}-S(h)_{1,s})-\kappa S(h)_{1,s} +\kappa)ds\\
 &\quad - \int_0^t \beta S(h)_{1,s}S(h)_{2,s}\dot{h}_sds,\\
S(h)_{2,t}&=x_2+ \int_0^t (\alpha S(h)_{1,s}S(h)_{2,s}+\frac{\beta^2}{2}(S(h)_{1,s}S(h)_{2,s})(S(h)_{2,s}-S(h)_{1,s})-(\gamma +\kappa) S(h)_{2,s})ds\\
& \quad+ \int_0^t\beta S(h)_{1,s}S(h)_{2,s}\dot{h}_sds,\\
S(h)_{3,t}&=x_3+ \int_0^t (\gamma S(h)_{2,s}-\kappa S(h)_{3,s})ds.\\
\end{aligned}
\end{equation}
 Then supp$(\mathbb{P}\circ X^{-1})=\overline{\mathcal{S}}$, where $\overline{\mathcal{S}}$ denotes the closure of $\mathcal{S}=\{S(h);h\in\mathcal{H}\}$ in $\mathcal{C}([0,1];\mathbb{R}^3)$.
\end{proposition}
\textbf{Example 4.5}(Threshold Ornstein-Ulenbeck processes) 
\begin{equation}
dX_t=\sum_{i=1}^{n}(\beta_i -\alpha_i X_t)I_{\{\theta_{i-1}\leq X_t <\theta_i\}}dt + \sigma d B_t,\\
\label{exp3.4}
\end{equation}
where $ \beta_1,\cdots,\beta_n,\alpha_1,\cdots,\alpha_n,-\infty=\theta_0<\theta_1<\cdots<\theta_{n-1}<\theta_n=\infty$ are constants. Set $b(x)=\sum_{i=1}^{n}(\beta_i -\alpha_ix)I_{\{\theta_{i-1}\leq x < \theta_i\}}$ and $\sigma(x)=\sigma$. It's clear that (\ref{exp3.4}) satisfy the Assumption \ref{assump2} with $V(x)=|x|^2.$ Then applying Thm \ref{support}, we have:
\begin{proposition}
Let $\mathcal{H}$ denote the Cameron-Martin space. For $h\in \mathcal{H}$, let $S(h)$ be the solution of the differential equation:
\begin{equation}\nonumber
\begin{aligned}
S(h)_t&=X_0+ \int_0^t \sum_{i=1}^{n}(\beta_i -\alpha_i S(h)_s)I_{\{\theta_{i-1}\leq S(h)_s < \theta_i\}}ds + \int_0^t\sigma\dot{h}_sds.\\
\end{aligned}
\end{equation}\nonumber
 Then supp$(\mathbb{P}\circ X^{-1})=\overline{\mathcal{S}}$, where $\overline{\mathcal{S}}$ denotes the closure of $\mathcal{S}=\{S(h);h\in\mathcal{H}\}$ in $\mathcal{C}([0,1];\mathbb{R})$.\\
\end{proposition}

\section{Acknowledgement}
\setcounter{equation}{0}
 \setcounter{definition}{0}

This work is partially supported by National Key R$\&$D program of China (No. 2022 YFA1006001)), National Natural Science Foundation of China (Nos. 12131019, 12371151, 11721101). Jianliang Zhai's research is also supported by the School Start-up Fund(USTC) KY0010000036 and the Fundamental Research Funds for the Central Universities(No. WK3470000016).

\end{document}